\documentclass[12pt,letterpaper]{amsart}
\usepackage{lipsum}
\usepackage[scr=rsfs]{mathalpha}
\usepackage{xspace}
\usepackage{amsfonts}
\usepackage{amsthm}
\usepackage{amssymb}
\usepackage{times} 
\usepackage{graphicx}
\usepackage{hyperref}
\usepackage{tikz}
\usepackage{comment}
\usetikzlibrary{arrows}
\usepackage[all]{xy}
\usepackage{tikz-cd}
\usepackage{enumerate}
\usepackage{mathbbol}
\usepackage{verbatim}
\usepackage{mathtools}

\date{}
\usepackage{mathtools}

\newtheorem{thm}{Theorem}

\newtheorem*{thm*}{Main Theorem}
\newtheorem*{lemma*}{Key Lemma}
\tikzset{node distance=5cm, auto}

\newtheorem{lemma}[thm]{Lemma}

\newtheorem{corollary}[thm]{Corollary}

\newtheorem{remark}[thm]{Remark}

\makeatletter
\newcommand{\etale}{\'etal\@ifstar{\'e}{e\xspace}}
\makeatother
\usepackage[OT2,T1]{fontenc}
\DeclareSymbolFont{cyrletters}{OT2}{wncyr}{m}{n}
\DeclareMathSymbol{\Sha}{\mathalpha}{cyrletters}{"58}
\usepackage{amssymb}

\begin{document}
\title[${u}$-invariants of perfectoid fields]{Which numbers are  $u$-invariants of perfectoid fields?}
\author{Eknath Ghate}
\address{School of Mathematics\\
Tata Institute of Fundamental Research\\
Mumbai\\
India}
\email{eghate@math.tifr.res.in}
\author{Srinivasan Srimathy}
\address{School of Mathematics\\
Tata Institute of Fundamental Research\\
Mumbai\\
India}
\email{srimathy@math.tifr.res.in}
\begin{abstract}
    Let $p\neq 2$ be a prime. We determine the set of numbers that occur as  the $u$-invariant of perfectoid subfields of $\mathbb{C}_p$. Given any number that is the $u$-invariant of a field of characteristic $p$,  we give an explicit construction of a  perfectoid field  with the same $u$-invariant. In particular, if a number is the $u$-invariant of  a field of characteristic $p$, it is also the $u$-invariant of a field of characteristic zero.
\end{abstract}
\maketitle
\section{Introduction}
All the quadratic forms considered here are non-degenerate and none of the fields  in this paper are  of characteristic $2$.   For a field $F$, the $u$-\emph{invariant} of $F$, denoted $u(F)$,  is defined to be the largest dimension of an anisotropic  quadratic form over $F$.  The  set of natural numbers that can be realized as the  $u$-invariant of some field is not completely known. Here is a brief history. Merkurjev established the existence of fields with any even $u$-invariant (\cite{merkurjev}). On the other hand, the numbers $3$, $5$, or $7$ cannot occur as the $u$-invariants of any field (\cite[Chapter XI, Proposition 6.8]{lam}). The situation for odd values was further clarified when Izhboldin  constructed a field of $u$-invariant $9$ (\cite{izhboldin}), and Vishik  later extended this by proving the existence of fields with $u$-invariant $2^r + 1$ for all $r \geq 3$ (\cite{vishik}). There has been some very recent progress in this direction due to Karpenko   where he shows that any number not of the form $2^r-1$ or $2^r-3$ is the $u$-invariant of some field (\cite{karpenko2}). Looking across different characteristics, it is an open question  whether a number is the $u$-invariant of a field of characteristic zero if and only if it is the $u$-invariant of a field of characteristic $p$ for some (or every) $p$. In this paper, we prove the "if" direction. In fact, the characteristic zero fields that we construct for this purpose are perfectoids.\\
 \indent Let $K$ be a Henselian valued field with residue field $k$ of characteristic $p \neq 2$ and value group $\Gamma_K$ (in additive notation).  Then it is well known (see for example, \cite[Theorem 4]{pumplun}) that 
    \begin{align}\label{eqn:uinv}
        u(K) = |\Gamma_K/2\Gamma_K| u(k).
    \end{align}

A \emph{perfectoid field} (\cite[Definition 3.1]{scholze}) is a complete non-archimedean field $K$ of residue
characteristic $p>0$  whose associated rank-$1$-valuation is non-discrete, such that the Frobenius 
\begin{align*}
    \mathrm{Fr}: \mathcal{O}_K/p &\rightarrow \mathcal{O}_K/p\\
    x &\mapsto x^p
\end{align*}
 is surjective. 

In this note, we study the $u$-invariant of perfectoid fields.  From (\ref{eqn:uinv}), computing $u$-invariant of a perfectoid field $K$  reduces to computing  $|\Gamma_K/2\Gamma_K|$   and $u(k)$.  As a consequence,  we get:
\begin{remark} \label{rmk:tilt} \normalfont
    The $u$-invariant of a perfectoid field and its tilt are the same as they share the same value group and residue field (\cite[Lemma 3.4 (iii), Proposition 3.6]{scholze}).
\end{remark}
 While it is direct to compute the $u$-invariant of a perfectoid field using (\ref{eqn:uinv}), it is an open question whether there is a perfectoid whose $u$-invariant is the given number.  In this paper, we show that given any number $n$ that is the $u$-invariant of a field of characteristic $p$ for \emph{some} $p\neq 2$, there exists a perfectoid field of characteristic zero whose $u$-invariant is $n$. Our proof is constructive and we provide an explicit construction of a perfectoid field for each such  $n$.  En route to this main result, we also determine the precise set of numbers that occur as $u$-invariants  of perfectoids contained in $\mathbb{C}_p$.

\section{Perfectoids contained in $\mathbb{C}_p$}
Let $\mathbb{C}_p$ denote the completion of the algebraic closure of the $p$-adic numbers. We begin with some examples:

\begin{enumerate}
    \item Let $\mathbf{K = \mathbb{C}_p}$: \\
    We have $\Gamma_K = \mathbb{Q}$, $|\Gamma_K/2\Gamma_K|=1$ and $u(\overline{\mathbb{F}_p}) =1 $. So $u(K) =1$. One can also directly see this since the $u$-invariant of any algebraically closed field is $1$. 
     \item Let $\mathbf{K = \widehat{\mathbb{Q}_{p}(p^{1/p^{\infty}}}})$:\\
    We have $\Gamma_K = \mathbb{Z}[1/p]$, $\Gamma_K/2\Gamma_K=\mathbb{Z}/2\mathbb{Z}$ and $u(k)=u(\mathbb{F}_p) =2$. So $u(K) =4$.
    \item Let $\mathbf{K = \hat{L}_{\infty}}$: (\cite[Proposition 1.4.12]{schneider})\\
    We have $\Gamma_K = \mathbb{Z}[1/p]$, $\Gamma_K/2\Gamma_K=\mathbb{Z}/2\mathbb{Z}$ and $u(k)=u(\mathbb{F}_q) =2$. So $u(K) =4$
  \end{enumerate}

\noindent To compute the $u$-invariants of an arbitrary perfectoid contained in $\mathbb{C}_p$, we start with the following lemma:
\begin{lemma}\label{lem:subQ}
   Let $\Gamma$ be a subgroup of  $(\mathbb{Q},+)$. Then $\Gamma/2\Gamma$ is either trivial or isomorphic to $\mathbb{Z}/2\mathbb{Z}$.
\end{lemma}
\begin{proof}
    Suppose $\Gamma$ is $2$-divisible then $\Gamma/2\Gamma$ is trivial. Else, there exists $a \in \Gamma\setminus  2\Gamma$. Now let $b\in \Gamma$ be any element. Since $a,b \in \mathbb{Q}$, there exists $m,n \in \mathbb{Z}$ with $(m,n)=1$ such that 
    \begin{align*}
        ma = nb,
    \end{align*}
    Suppose $n$ is even, then $m=2r+1$ is odd. This implies that $a = nb -2r a \in 2 \Gamma$ contradicting the hypothesis on $a$. Therefore $n = 2s+1$ is odd. Now we have, 
    \begin{align*}
        b=ma- 2sb \in \begin{cases} \Gamma  \text{~if $m$ is even}\\
        a + 2\Gamma \text{~else.}\end{cases}
    \end{align*}
    Therefore $\Gamma/2\Gamma \cong \mathbb{Z}/2\mathbb{Z}$.
\end{proof}

\begin{thm}\label{thm:ucp}
    Let $\mathbf{\mathbb{Q}_p \subseteq K \subseteq \mathbb{C}_p}$ be an arbitrary perfectoid. Then $u(K)\in\{1,2,4\}$ and each of these values is the $u$-invariant of a perfectoid in $\mathbb{C}_p$.
\end{thm}
First we show that  the only possible $u$-invariants of $K$ are $1,2$ or $4$. Since residue $k$ of $K$ is an algebraic extension of $\mathbb{F}_p$, we have  $u(k)\leq 2$ (since $k$ is a $C_1$-field).  Therefore  by (\ref{eqn:uinv}), it suffices to show that  $|\Gamma_K/2\Gamma_K| \leq 2$. Since $\Gamma_{\mathbb{C}_p} = \mathbb{Q}$, $\Gamma_K$ is a subgroup of $(\mathbb{Q}, +)$.  The claim follows from Lemma \ref{lem:subQ}.\\
\indent Now we show that each of the values $\{1,2,4\}$ is the $u$-invariant of a perfectoid $K \subseteq \mathbb{C}_p$. The cases  $u(K)=1,4$ are already handled in the examples above. Consider the field
\begin{align*}
    K = \widehat{\mathbb{Q}_p(p^{1/2^{\infty}}, p^{1/p^{\infty}})}.
\end{align*}
By \cite[Chapter II, (4.3)]{neukirch}, $$\mathcal{O}_K/p \cong \mathbb{Z}_p[p^{1/2^{\infty}}, p^{1/p^{\infty}}]/p . $$
    Consider the Frobenius map
    \begin{align*}
        \mathrm{Fr}: \mathcal{O}_K/p  \rightarrow \mathcal{O}_K/p  .
    \end{align*}
    Since $\mathbb{Z}_p/p \cong \mathbb{F}_p$ is perfect and $\mathcal{O}_K/p $ is a ring of characteristic $p$, in order to  check surjectivity of $\mathrm{Fr}$, it suffices to check if $p^{1/2^n }$  has $p$-th root in the ring i.e., we just need to check if $p^{1/2^np } \in \mathcal{O}_K$ for every $n$. Since $p\neq 2$, there are $m, r \in \mathbb{Z}$ such that $m\cdot2^n + r\cdot p = 1$. Hence  
\begin{align*}
       p^{1/2^np }= p^{(m\cdot2^n + r\cdot p)/2^np} =  p^{m/p + r/2^n}\in \mathcal{O}_K .
    \end{align*}
The value group  of $K$ is $\Gamma_K \cong \mathbb{Z}[1/2, 1/p]$ which is dense in $(\mathbb{R}, +)$. Moreover, $|\Gamma_K/2\Gamma_K| = 1$ since $\Gamma_K$ is $2$-divisible. Therefore $K$ is a perfectoid with residue $\mathbb{F}_p$ and $u(K)=2$ by (\ref{eqn:uinv}).

\begin{remark}\normalfont
    From the above discussion it is clear that for a perfectoid field $K \subseteq \mathbb{C}_p$, $u(K) =1$ precisely when $\Gamma_K$ is $2$-divisible and the residue field $k$ is quadratically closed. Similarly, $u(K)=4$ precisely when $\Gamma_K$ is not $2$-divisible and $k$ is not quadratically closed. For the other cases, $u(K)=2$. 
\end{remark}

\section{Perfectoid fields with  given $u$-invariant} 
In this section, we show that any number that can be realized as the $u$-invariant of a field of characteristic $p\neq 2$ can be realized as the $u$-invariant of some perfectoid field.  For a field $\kappa$, let $\kappa^{\mathrm{perf}}$ denote the perfect closure of $\kappa$ and let $W(\kappa)$ denote the Witt ring of $\kappa$.
\begin{lemma}\label{lem:uperf}
Let $\kappa$ be a field of characteristic $p\neq 2$.  Then
\begin{enumerate}[(i)]
    \item the restriction map
    \begin{align*}
        r_{\kappa^{\mathrm{perf}}/\kappa}: W(\kappa) &\rightarrow W(\kappa^{\mathrm{perf}})\\
        q &\mapsto q\otimes_\kappa \kappa^{\mathrm{perf}}\nonumber
    \end{align*}
    is an isomorphism.
    \item $u(\kappa) = u(\kappa^{\mathrm{perf}})$.
\end{enumerate}
\end{lemma}
\begin{proof}
    \begin{enumerate}[(i)]
        \item First we show surjectivity of the restriction map. Let $q= \langle a_1, a_2, \cdots, a_n\rangle \in W(\kappa^{\mathrm{perf}})$. Since each $a_i \in (\kappa^{\mathrm{perf}})^{\times}$, $a_i^{p^m}\in \kappa^{\times}$ for some $m$.   Since $p$ is odd, $a_i^{p^m}= a^{2r+1}= a_i^{2r}.a_i \in \kappa$ and hence $a_i = b_i c_i^{2}$ for some $b_i \in \kappa$ and $c_i \in (\kappa^{\mathrm{perf}})^{\times}$.  Therefore  $q = r_{\kappa^{\mathrm{perf}}/\kappa}(q')$  where $q'=\langle b_1,\cdots, b_n\rangle$. \\
        \indent To show injectivity,  suppose $q$ is an $n$-dimensional form over $\kappa$ such that $q \otimes_{\kappa} \kappa^{\mathrm{perf}}$ is  equivalent to a hyperbolic form.  Then  entries of the $n\times n$ matrix  giving this equivalence come from a finite  purely inseparable extension $\kappa'/\kappa$ and hence  $q \otimes_{\kappa} \kappa'$ is hyperbolic. Now the fact that $[\kappa':\kappa]$ is odd together with \cite[Chapter VII, Corollary 2.6]{lam} implies that $q$ is hyperbolic.
        
        \item The surjectivity of  the restriction map $r_{\kappa^{\mathrm{perf}}/\kappa}$ directly implies that $u(\kappa^{\mathrm{perf}}) \leq u(\kappa)$. Suppose $u(\kappa^{\mathrm{perf}}) < u(\kappa)$.  This implies that every non-degenerate quadratic form of dimension $u(\kappa)$ over $\kappa^{\mathrm{perf}}$ is isotropic.  Consider an anisotropic quadratic form $q$ over $\kappa$ of dimension $n:=u(\kappa)$. By the above argument, $q \otimes_{\kappa} \kappa^{\mathrm{perf}}$  is isotropic. So there exists $\mathbf{0} \neq \mathbf{a} \in (\kappa^{\mathrm{perf}})^n$ such that  $q(\mathbf{a}) =0$. The entries in the vector $\mathbf{a}$ come from a finite purely inseparable extension $\kappa'/\kappa$. Since $p$ is odd, $[\kappa':\kappa]$ is of odd degree and the above arguments imply that $q \otimes_{\kappa} \kappa'$ is isotropic. This contradicts Springer's theorem that any anisotropic form stays anisotropic over odd degree extensions (\cite[Chapter VII, Theorem 2.7]{lam}). Therefore $u(\kappa^{\mathrm{perf}}) = u(\kappa)$.
    \end{enumerate}
\end{proof}
The following theorem and its proof explicitly constructs a perfectoid field of characteristic zero whose $u$-invariant is the $u$-invariant of some field of positive characteristic.
\begin{thm}\label{thm:uperf}
    Let  $\kappa$ be any field  of  positive characteristic $p \neq 2$. Then there exists a perfectoid field $K$ of characteristic zero with $u(K) = u(\kappa)$. Moreover, one can construct $K$ in  such a way that the residue of $K$ is $\kappa^{\mathrm{perf}}$.
\end{thm}
\begin{proof}
    Let $R$ be the absolutely unramified complete discrete valuation ring with residue field $\kappa^{\mathrm{perf}}$. Such a ring exists and is unique upto isomorphism by \cite[Chapter II, \S 5, Theorem 3]{serre_local}. Let $F$ be the fraction field of $R$ so that $R= \mathcal{O}_F$. Let 
    \begin{align*}
    L &= F(p^{1/2^{\infty}}, p^{1/p^{\infty}})  \text{~~~and}\\
    K &= \hat{L}
    \end{align*}
By \cite[Chapter I, \S6(ii)]{serre_local},  $\mathcal{O}_L = \mathcal{O}_F[p^{1/2^{\infty}}, p^{1/p^{\infty}}]$. Moreover, by \cite[Chapter II, (4.3)]{neukirch}, $$\mathcal{O}_K/p \cong \mathcal{O}_F[p^{1/2^{\infty}}, p^{1/p^{\infty}}]/p . $$
    To show that $K$ is  a perfectoid, it suffices to show that  the Frobenius map
    \begin{align*}
        \mathrm{Fr}: \mathcal{O}_K/p  \rightarrow \mathcal{O}_K/p 
    \end{align*}
    is surjective. Since $\mathcal{O}_F/p \cong \kappa^{\mathrm{perf}}$ is perfect and $\mathcal{O}_K/p $ is a ring of characteristic $p$, in order to  check surjectivity of $\mathrm{Fr}$, it suffices to check if $p^{1/2^n}$  has $p$-th root in the ring for every $n$. The proof goes similar to the corresponding proof in Theorem \ref{thm:ucp}.\\
    \indent The value group  of $K$ is $\Gamma_K \cong \mathbb{Z}[1/2, 1/p]$  and $|\Gamma_K/2\Gamma_K| = 1$. It is easy to see that the residue field of $K$ is $\kappa^{\mathrm{perf}}$. Therefore by (\ref{eqn:uinv}) and Lemma \ref{lem:uperf},  $u(K) = u(\kappa^{\mathrm{perf}}) = u(\kappa)$. 
    \indent 
\end{proof}
The following corollaries easily follow from Theorem \ref{thm:uperf}.

\begin{corollary} \label{cor:char0}
    Given  any field $\kappa$ of characteristic $\neq 2$, there exists a complete non-discrete rank one valued field  $K$ of characteristic zero with $u(K) = u(\kappa)$.
\end{corollary}
\begin{proof}
     When characteristic of $\kappa$ is positive, the claim follows directly from Theorem \ref{thm:uperf}. Now suppose $\kappa$ is of characteristic zero. The field $K= \widehat{\kappa((x^{1/2^{\infty}}))}$  is a complete  non-discrete rank one valued  field with residue $\kappa$ and whose $\Gamma_K$ is $2$-divisible. Therefore  $u(K)=u(\kappa)$ by (\ref{eqn:uinv}). 
\end{proof}

\begin{corollary}
    For any  $u \in \mathbb{N}$ not of the form $2^r-1$ or $2^r-3$, there exist a perfectoid field of every characteristic $\neq 2$ whose $u$-invariant is $u$.
\end{corollary}
\begin{proof}
    This follows from  \cite[Theorem 1.1]{karpenko2}, Theorem \ref{thm:uperf} and  Remark \ref{rmk:tilt}.
\end{proof}

\section*{Acknowledgements}
The authors acknowledge the support of the DAE, Government of India, under Project Identification No.~RTI4014. They would like to thank K{\k{e}}stutis {\v{C}}esnavi{\v{c}}ius, Nikita Karpenko, Sandeep Varma, Suresh Venapally  and Alexander Vishik for very helpful discussions and valuable suggestions.

\nocite*{}
\bibliographystyle{alpha}
\bibliography{ref}
\end{document}